\documentclass{article}

\usepackage{lineno,hyperref}
\modulolinenumbers[5]

\usepackage{amsmath, amsthm, amssymb}
\usepackage[ansinew]{inputenc}
\usepackage{amsfonts}
\usepackage{amssymb}
\usepackage{enumitem}
\usepackage{amsthm}
\usepackage{dsfont}
\usepackage{natbib}
\bibpunct[, ]{[}{]}{,}{n}{,}{,}
\makeatother
\theoremstyle{plain}
\DeclareMathAlphabet{\mathbbmsl}{U}{bbm}{m}{sl}
\def\diag{\mathop{\rm diag}}
\def\min{\mathop{\rm min}}

\def\d{{\bf d}}
\def\cs{cS-$W$}

\def\Rr{\mathbb{R}^n}

\def\min{\mathop{\rm min}}

\def\R{\mathbb{R}^{n\times n}}

\def\x{{\bf{x}}}
\def\y{{\bf{y}}}

\newtheorem{theorem}{Theorem}[section]
\newtheorem{lemma}[theorem]{Lemma}
\newtheorem{corollary}[theorem]{Corollary}
\newtheorem{proposition}[theorem]{Proposition}
\newtheorem{example}[theorem]{Example}

\newtheorem{rem}{Remark}
\newtheorem{definition}{Definition}
\addcontentsline{toc}{chapter}{Bibliography}
\begin{document}
	\begin{center}
	{\bf	On Column sufficiency  and Extended Horizontal Linear Complementarity Problem}
	
		\textsc{Punit Kumar Yadav$^{a}$, K. Palpandi$^{b}$}\\
	$^{a}$Department of Mathematics, JECRC University, 303905, India.\\
	
	$^{b}$Department of Mathematics, National Institute of Technology Calicut, 673601, India.\\
	E-mail address: punitjrf@gmail.com\\\end{center}
	\begin{abstract}
In this article, we introduce the concept of the column-sufficient $W$-property for a set of matrices and prove the convexity of the solution set for the Extended Horizontal Linear Complementarity Problem. Additionally, we present an uniqueness result and establish several results related to the column-sufficient $W$-property.
	\end{abstract}

{\it MSCcodes}
15B99, 90C33.
\section{Introduction}
A real $n \times n$ matrix $C$ is a column-sufficient matrix if 
$$x \in \mathbb{R}^n, ~x * Cx \leq 0 \implies x * Cx = 0,$$ 
where $'*'$ denotes the pointwise product. Initially, Cottle et al. \cite{sufficient} introduced the concept of a column-sufficient matrix, and proved that the class of column-sufficient matrices ensures the convexity of the solution set of linear complementarity problem (LCP). The LCP($C,q$) is to find a vector $x\in\Rr$ such that 
\begin{equation*}
	x\in\Rr_{+},~y = Cx + q \in \mathbb{R}^n_+ \text{,~and}~ x * y = 0,
\end{equation*}
where $C\in\mathbb{R}^{n\times n}$ and $q\in\Rr$. The LCP has wide-ranging applications in optimization, economics, and game theory, see \cite{sufficient,LCP,elcp,PP0}. Over the past three decades, many generalizations of the LCP have been studied in the literature (see \cite{wcp,telcp,elcp}). Many important properties have been explored in relation to the class of column-sufficient matrices, see \cite{sufficient,LCP}. For example, the $P$ matrix property ($[x \in \mathbb{R}^n, ~x * Cx \leq 0 \implies x = 0]$) is equivalent to the column-sufficient matrix property under the nondegeneracy condition ($[x \in \mathbb{R}^n, ~x * Cx = 0 \implies x = 0]$). For the class of $Z$-matrices, the column-sufficient matrix property is also equivalent to the column-sufficient matrix property on $\Rr_+$, see in \cite{LCP}.

Mangasarian and Pang \cite{elcp} has expanded the concept of the column sufficient property to the X-column-sufficiency property $([x_0*x_1\leq0,~ C_0x_0-C_1x_1=0\implies (x_0,x_1)=(0,0)])$ within the context of horizontal linear complementarity problem (HLCP). Furthermore, Gowda \cite{glcp} introduced and characterized column-sufficiency and $P$ properties specific to vertical, horizontal, and mixed LCPs. He also proves that the X-column-sufficiency property is equivalent the convexity of solution set for HLCP. For a pair of real $n\times n$ matrices $C_0$ and $C_1$, along with a vector $q\in\Rr$, the HLCP seeks a pair of vectors $(x,y)\in\Rr\times\Rr$ that satisfy the conditions:
\begin{equation*}
	\begin{aligned}
		C_0 x-C_1y=q~~\text{and}~~x\wedge y=0.
	\end{aligned}
\end{equation*}

The HLCP has gained significant research attention due to its extensive applications across various domains (see references \cite{telcp,homo,elcp}). The HLCP is generalized, in context to the set of matrices, as the Extended horizontal linear complementarity problem, see \cite{glcp,szn,PP0}. For an ordered set of matrices ${\bf C}:=\{C_0,C_1,...,C_k\} \subseteq \mathbb{R}^{n \times n}$, a vector $q\in \mathbb{R}^n$, and an ordered set of positive vectors ${\bf d}:=\{d_{1},d_{2},...,d_{k-1}\} \subseteq \mathbb{R}^n$, the Extended Horizontal Linear Complementarity Problem, denoted EHLCP(${\bf C},{\bf d},q$), aims to find vectors $x_{0},x_{1},...,x_{k} \in \mathbb{R}^n$ that satisfy the conditions:\begin{equation*}\label{e1}
	\begin{aligned}
		C_0 {{x}_{0}} =&q+\sum_{i=1}^{k} C_i {x}_{i},\\ 
		~~ {x}_{0} \wedge {x}_{1}=0 ~~\text{and}~~ (d_{j}-&{x}_{j})\wedge {x}_{j+1}=0 ~\forall j\in[k-1].\\
	\end{aligned}
\end{equation*}
\noindent When $k=1$, EHLCP reduces to the HLCP. Furthermore, if we set \( C_0 = I,k=1 \), the EHLCP becomes to the standard LCP. A comprehensive analysis of the EHLCP can be found in the article by Sznajder and Gowda \cite{PP0}. Recently, Yadav and Palpandi \cite{pk,cnd} have extended the concepts of \( R_0 \), SSM, and nondegenerate matrix properties in the context of the EHLCP to study the solution set of the EHLCP.

Motivated by the column-sufficient property in context with LCP and HLCP, we pose the question: Can we define column-sufficient complementarity conditions for sets of matrices and demonstrate the convexity of the solution set for the EHLCP? To address this, we introduce the notion of the column sufficient-$W$ property for sets of matrices. Our contributions include the following:
\begin{itemize}
	\item[(i)] We define the column sufficient-$W$ property and prove the convexity of the solution set for EHLCP.
	\item[(ii)] We prove that the EHLCP has a unique solution for $q>0$ due to the column sufficient-$W$ property. 
	\item[(ii)] We present relationship between  the column $W$-property and the column sufficient-$W$ property. Also, we prove a result related to $Z$-matrices in context of column sufficient-$W$ property.
	
\end{itemize}
The structure of the paper is as follows: Section 2 presents essential definitions and  necessary results for subsequent discussions. In next section, we introduce the column sufficient-\( W \)property and establish our main theorem. The final section examines some properties related to the the column \( W \)-property and the column sufficient-\( W \) property.
\section{Notation and Preliminaries}
\subsection{Notation}Throughout this paper, we use the following notation: \begin{itemize}
	\item[(i)] $\mathbb{R}^n$ denotes the $n$-dimensional Euclidean space with the usual inner product.  The set of all nonnegative vectors (resp., positive vectors) in $\mathbb{R}^n$ is denoted by $\mathbb{R}^n_+$ (resp., $\mathbb{R}^n_{++}$). We write $x \geq 0$ (resp., $x > 0$) if and only if $x \in \mathbb{R}^n_+$ (resp., $x \in \mathbb{R}^n_{++}$).
	\item [(ii)] The notation $[n]$ represents the set $\{1, 2, \ldots, n\}$.
	\item[(iii)]  For $u, v \in \mathbb{R}$, $u * v$ denotes the pointwise product of $u$ and $v$, i.e., $(u * v)_i = u_i v_i$ for all $i \in [n]$. Similarly, $u \wedge v$ denotes the pointwise minimum of $u$ and $v$, i.e., $(u \wedge v)_i = \min\{u_i, v_i\}$ for all $i \in [n]$.
	\item [(iv)] The $k$-ary Cartesian power of $\Rr$ (resp., $\Rr_{++}$) is denoted by $\Lambda^{(k)}_n$ (resp., $\Lambda^{(k)}_{n,++}$). The bold zero $\mathbf{0}$ represents the zero vector $(0, 0, \ldots, 0) \in \Lambda^{(k)}_n$.
	
	\item [(v)] The set of all $n\times n$ real matrices denotes $\R$. The symbol $\Lambda^{(k)}_{n\times n}$ is used to denote the $k$-ary Cartesian product of  $\R$. 
	\item [(vi)] For a matrix $A\in\R$,  $\text{diag}(A)=(A_{11},A_{22},...,A_{nn})\in \Rr$ denotes a vector in $\Rr$ , where $A_{ii}$ is the $ii^{\rm th}$ diagonal entry of matrix $A$ and $(A)_{ij}$ denotes the $ij^{th}$ entry of matrix $A$.  $\text{det}(A)$ denotes the determinant of matrix $A$.
	\item[(vii)] SOL(${\bf C}, \d, q$) represents the set of all solutions to EHLCP($\mathbf{C}, \d, q$).
\end{itemize} 
We now recall some definitions and results from the LCP theory for later use.

\begin{proposition}[\cite{wcp}]\label{msd}
	Let $x,y\in\mathbb{R}^n.$ Then, the following statements are equivalent.
	\begin{itemize}
		\item [\rm(i)] $x\wedge y=0.$
		\item [\rm(ii)] $x,y\geq 0$ and $~x*y=0,$ where $*$ is the Hadamard product.
		\item[\rm(iii)] $x,y\geq 0~\text{and}~\langle x,y\rangle=0.$
\end{itemize}	\end{proposition} 
\begin{definition}\cite{LCP}
	We say a matrix $C\in\mathbb{R}^{n\times n}$ is a $M$-matrix if its  off-diagonal elements are nonpositive and it has a nonnegative inverse.
\end{definition}
Also if $C$ is $M$-matrix, then  $x>0$ implies $C^{-1}x>0.$
\begin{lemma}[{Lemma 1}, \cite{pk}]\label{l1}   Let ${\bf C}=(C_0,C_1,...,C_k) \in \Lambda^{(k+1)}_{n\times n}$ and $\x =(x_{0},x_{1},...,x_{k})\in\text{\rm SOL}({\bf C},\d,q)$. Then $\x$ satisfies the following system $$C_0x_{0}=q+\sum_{i=1}^{k} C_ix_{i}~\text{and}~x_{0}\wedge x_{j}=0~\forall j\in [k].$$  
\end{lemma}
\begin{definition}[\cite{PP0}]\rm
	Let ${\bf C}=(C_0,C_1,...,C_k)\in\Lambda^{(k+1)}_{n\times n}$. Then a matrix $R\in\R$ is a {\it column representative} of ${\bf C}$ if $$R._j\in\big\{(C_0)._j,(C_1)._j,...,(C_k)._j\big\},~\forall j\in[n],$$
	where $R._j$ is the $j^{{\rm th}}$ column of matrix $R.$
\end{definition}
\begin{rem}\rm\label{Represent}
	From the above definition, we observe that each column representative is of the form \(C_0 I_0 + C_1 I_1 + \cdots + C_k I_k\), where the entries of \(I_i~\forall i\in\{0,1,2,...,k\}\) are only \(0\) or \(1\), satisfying \(I_i I_j = 0\) for all \(0 \leq i < j \leq k\) and \(\text{diag}(I_0 + I_1 + \cdots + I_k) \neq 0\). This characterization helps in understanding the structure of column representative matrices.
\end{rem}
Next, we recall the column $W$-property.
\begin{definition}[\cite{PP0}]\label{cw0} \rm
	Let  ${\bf C}:=(C_0,C_1,...,C_k)\in\Lambda^{(k+1)}_{n\times n}$. Then we say that ${\bf C}$ has the 
	\begin{itemize}
		\item[\rm (i)] {\it column $W$-property} if the determinants of all the column representative matrices of ${\bf C}$ are all positive or all negative.
		\item[\rm (ii)] {\it column $W_0$-property} if the determinants of all the column representative matrices of ${\bf C}$ are all nonnegative (nonpositive), and at least one column representative matrix has a positive(negative) determinant.
		
	\end{itemize}  
\end{definition}
Due to Sznajder and Gowda \cite{PP0}, we have the following results.
\begin{theorem}\rm[\cite{PP0}] \label{P1}
	For ${\bf C}=(C_0,C_1,...,C_k)\in\Lambda^{(k+1)}_{n\times n}$, the following are equivalent:\begin{itemize}
		
		\item[\rm(i)]${\bf C}$ has the column $W$-property.
		\item[\rm(ii)] For arbitrary nonnegative diagonal matrices $D_{0},D_{1},...,D_{k}\in\R$
		with $\text{\rm diag}(D_{0}+D_{1}+...+D_{k})>0$,
		$$\text{\rm det}\big(C_0D_{0}+C_1D_{1}+...+C_kD_{k}\big)\neq 0.$$
		\item [\rm(iii)]$C_0$ is invertible and  $(I,C_0^{-1}C_1,...,C_0^{-1}C_k)$ has the column $W$-property.
		\item [\rm(iv)]  For all $q\in\Rr$ and $\d\in\Lambda^{(k-1)}_{n,++}$, {\rm EHLCP}$({\bf C},\d,q)$ has a unique solution. \end{itemize}
\end{theorem} 
In next we recall the definition of column ND-$W$ property from \cite{cnd}.
\begin{definition}\rm
	We say that  ${\bf C}=(C_0,C_1,...,C_k)\in\Lambda^{(k+1)}_{n\times n}$ has the {\it column~nondegenerate} -$W$(column ND-$W$) property if 
	\begin{equation*}
		\begin{aligned}
			\left. \begin{array}{r} C_0x_{0}=\displaystyle{\sum_{i=1}^{k} C_ix_{i}},
				\\x_i*x_j=0~\forall 0\leq i<j\leq k \end{array}\right\} \implies (x_0,x_1,...,{x_k})={\bf 0}. 
		\end{aligned}
	\end{equation*}
\end{definition}
For the finiteness of solution set of EHLCP, we have the following result.
\begin{theorem}[\cite{PP0}]\rm\label{cnd}
	Let ${\bf C}=(C_0,C_1,...,C_k)\in\Lambda^{(k+1)}_{n\times n}$.  Then the following statements are equivalent.
	\begin{itemize}
		\item [\rm(i)] ${\bf C}$ has the column ND-$W$  property.
		\item [\rm(ii)] {\rm EHLCP}$({\bf C},{\bf d},q)$ has only finitely many solutions for every ${\bf d}\in \Lambda^{(k-1)}_{n,++}$ and $q\in\Rr$.
	\end{itemize}
\end{theorem}
From theorems \ref{P1} and \ref{cnd}, we have a direct implication.
\begin{corollary}\label{cor1}\rm
	Let ${\bf C}=(C_0,C_1,...,C_k)\in\Lambda^{(k+1)}_{n\times n}$. If ${\bf C}$ has the column $W$-property, then ${\bf C}$ has the column ND-$W$  property.
\end{corollary}
\section{Column Sufficient-W Property}
In this section, we define the {\it column sufficient-$W$} property as a generalization of the column-sufficient matrix property. We then show that the column sufficient-$W$ property ensures the convexity of the solution set of the EHLCP.

\begin{definition}\label{csww}\rm
	We say ${\bf C}=(C_0,C_1,...,C_k)\in\Lambda^{(k+1)}_{n\times n}$ has column sufficient-W(cS-$W$) property if 
	\begin{equation*}\label{csw}
		\begin{aligned} \left. \begin{array}{r}x_i*x_j\geq 0~\forall1\leq i<j\leq k,\\ C_0x_{0}=\sum_{i=1}^{k} C_ix_{i},  x_0*x_i\leq 0~\forall i\in[k]\end{array}\right\} \Rightarrow x_i*x_{i+1}=0~\forall 0\leq i\leq k-1. \end{aligned}
	\end{equation*}
\end{definition}

\noindent For \( i = 1 \), we observe that the cS-$W$ property becomes the X-column-sufficiency property for a pair of square matrices as described in \cite{elcp}. Additionally, for \( i = 1 \) and \( C_0 = I \), it reduces to the column-sufficient matrix property.

Next we prove a relation between the cS-$W$ property and X-column-sufficiency property.
\begin{proposition}
	Let ${\bf C}$ has the {\rm cS}-$W$ property. Then, 
	the pair $(C_i,C_{i+1})$ has the X-column-sufficiency property for all $0\leq i\leq k-1.$
\end{proposition}
\begin{proof} We consider here two cases.
	
	{\it Case-I}: For $i=0,$ let $x_0,x_1\in\Rr$ such that  $$C_0x_0-C_1x_1=0\text{ and  }x_0*x_1\leq 0.$$Set $\x:=(x_0,x_1,0,...,0)\in\Lambda^{(k+1)}_{n}$. Since $x_i=0$  for $i\in\{2,3,...,k\}$, we have $x_i*x_j\geq 0~\forall1\leq i<j\leq k$ and  
	$$ C_0x_{0}=\sum_{i=1}^{k} C_ix_{i},~\text{and } x_0*x_i\leq 0~\forall i\in[k].
	$$
	As ${\bf C}$ has the \cs property, $x_i*x_{i+1}=0$ for all $0\leq i\leq {k-1}.$ This implies $x_0*x_1=0.$ Thus, $(C_0,C_1)$ has the X-column-sufficiency property. 
	
	{\it Case-II}: For $i\in[k-1]$, let $x_i,x_{i+1}\in\Rr$ such that  $$C_ix_i-C_{i+1}x_{i+1}=0\text{ and }x_i*x_{i+1}\leq 0.$$ Consider ${\bf v}:=(v_0,...,v_i,v_{i+1},...,v_k)=(0,0,...,0,x_i,-x_{i+1},0,...,0)\in\Lambda^{(k+1)}_{n}$. Since $x_i*x_{i+1}\leq 0$,  $(x_i)_r$ and $(-x_{i+1})_r$ will have same sign whenever $(x_i)_r*(x_{i+1})_r\neq 0$. Thus, $v_s*v_j\geq 0~\forall1\leq s<j\leq k$ and $v_0*v_s\leq 0~\forall s\in[k]$. This gives $v_s*v_{s+1}=0=x_s*x_{s+1}$ for all $0\leq s\leq {k-1}$. Hence, $(C_i,C_{i+1})$ has the X-column-sufficiency property for all $i\in[k-1]$.
\end{proof}
In LCP \cite{LCP}, if a matrix is column-sufficient, then the corresponding solution set of LCP is a convex set. Similarly, this result holds in the HLCP whenever the involved matrix pair has the X-column-sufficiency property, see in \cite{elcp}. We now prove that the column sufficient-\(W\) property gives the convexity of the solution set to EHLCP.
\begin{theorem}\label{convex}
	Let ${\bf C}=(C_0,C_1,...,C_k)\in\Lambda^{(k+1)}_{n\times n}$. If  ${\bf C}$ has the {\rm cS}-$W$ property, then  $\text{\rm SOL}({\bf C},\d,q)$ is convex for every $q\in\mathbb{R}^n$ and $\d \in  \Lambda^{(k-1)}_{n,++}$.
\end{theorem}
\begin{proof} Let $q\in\mathbb{R}^n$ and $\d \in  \Lambda^{(k-1)}_{n,++}$. If $\text{\rm SOL}({\bf C},\d,q)$ is empty or singleton, then the proof is complete. Suppose $\text{\rm SOL}({\bf C},\d,q)$ has at least two solutions; let $\x,\y\in \text{\rm SOL}({\bf C},\d,q)$, where $\x=(x_0,x_1,...,x_k)$ and $\y=(y_0,y_1,...,y_k).$ We now  claim that $t\x+(1-t)\y\in{\rm SOL}({\bf C},\d,q)~\forall 0\leq t\leq 1.$
	
	Since $\x,\y\in {\rm SOL}({\bf C},\d,q),$
	we have
	\begin{equation}\label{Cx}
		\begin{aligned}
			C_0 {{x}_{0}} =&q+\sum_{i=1}^{k} C_i {x}_{i},\\ 
			~~ {x}_{0} \wedge {x}_{1}=0 ~~\text{and}~~ (d_{j}-&{x}_{j})\wedge {x}_{j+1}=0 ~\forall j\in[k-1],\\
		\end{aligned}
	\end{equation}
	and
	\begin{equation}\label{y}
		\begin{aligned}
			C_0 y_{0}=&q+\sum_{i=1}^{k} C_iy_{i},\\
			y_{0}\wedge y_{1}=0 ~~\text{and} ~~ (d_{j}-&y_{j})\wedge y_{j+1}=0, ~\forall j\in[k-1].\\
		\end{aligned}
	\end{equation}   From  Lemma \ref{l1} and the equations above, we obtain 
	\begin{equation}\label{c1}
		C_0x_{0}=q+\sum_{i=1}^{k} C_ix_{i},~ x_0\wedge x_i=0~\forall i\in[k],
	\end{equation}
	and 
	\begin{equation}\label{c2}
		C_0y_{0}=q+\sum_{i=1}^{k} C_iy_{i},~ y_0\wedge y_i=0~\forall i\in[k].
	\end{equation}
	In view of Lemma \ref{msd}, from Eqs. (\ref{c1}) and (\ref{c2})
	\begin{equation}\label{c3}
		\begin{aligned}
			C_0(x_0-y_{0})=&\sum_{i=1}^{k} C_i(x_i-y_{i}),\\
			(x_0-y_0)*(x_i-y_i)&=-x_0*y_i-y_0*x_i\leq0~\forall i\in[k].
		\end{aligned}
	\end{equation}
	We now show  that $(x_i-y_i)*(x_j-y_j)\geq 0 ~\forall1\leq i<j\leq k$. Assume contrary. Then there exists $i^{*},j^{*}\in[k]$ with $i^{*}<j^{*}$ such that
	$$(x_{i^{*}}-y_{i^{*}})_r*(x_{j^{*}}-y_{j^{*}})_r< 0\text{ for some }r\in[n].$$ Hence $(x_{i^{*}}-y_{i^{*}})_r$ and $(x_{j^{*}}-y_{j^{*}})_r$ have opposite sign. Without loss of generality, we assume $(x_{i^{*}}-y_{i^{*}})_r>0$. Then $(x_{i^{*}})>(y_{i^{*}})_r$ which unifies $(d_{i^{*}})_r>(y_{i^{*}})_r$ due to Eqs \ref{Cx} and \ref{y} as $0\leq x_i,y_i\leq d_i~\forall i\in[k-1]$. From the complementarity condition, $(d_{i}-y_{i})_r*y_{i+1}=0~\forall i\in[k-1],$  $(y_{i+1})_r=(y_{i+2})_r=...=(y_{j})_r=0$. As $i^{*}<j^{*}$, $(x_{j^{*}}-y_{j^{*}})_r=(x_{j^{*}})_r>0$ which contradicts that $(x_{i^{*}}-y_{i^{*}})_r$ and $(x_{j^{*}}-y_{j^{*}})_r$ have opposite sign.
	Thus, \begin{equation}\label{wprop}
		(x_i-y_i)*(x_j-y_j)\geq 0 ~\forall1\leq i<j\leq k.\end{equation}
	Now, from Eqs.  (\ref{c3}) and (\ref{wprop}), we have
	\begin{equation*}
		\begin{aligned}
			&C_0(x_0-y_{0})=\sum_{i=1}^{k} C_i(x_i-y_{i}),\\
			(x_0-y_0)*(x_i-y_i)&\leq0~\forall i\in[k]\text{ and }
			(x_i-y_i)*(x_j-y_j)\geq 0 ~\forall1\leq i<j\leq k.
		\end{aligned}
	\end{equation*}As ${\bf C}$ has the \cs property, we get $(x_i-y_i)*(x_{i+1}-y_{i+1})=0~\forall 0\leq i\leq {k-1}.$
	From this, we have\begin{equation}\label{u0v1}
		(x_0-y_0)*(x_1-y_1)=-x_0*y_1-y_0*x_1=0\implies x_0*y_1=y_0*x_1=0,	\end{equation}
	and for $i\in[k-1]$, we have
	$	(x_i-y_i)*(x_{i+1}-y_{i+1})=0$  which implies that
	$$	((d_i-y_i)-(d_i-x_i))*(x_{i+1}-y_{i+1})=
	(d_i-y_i)*x_{i+1}+(d_i-x_i)*y_{i+1}=0.$$
	As $0\leq x_i,y_i\leq d_i~\forall i\in[k-1]$, both terms $(d_i-y_i)*x_{i+1}$, $(d_i-x_i)*y_{i+1}$  are nonnegative. So,\begin{equation}\label{di}
		(d_i-y_i)*x_{i+1}=(d_i-x_i)*y_{i+1}=0~\forall i\in[k-1].\end{equation}
	As $\x,\y\in{\rm SOL}({\bf C},\d,q),$ from Eqs. (\ref{c1}) and (\ref{c2}), we have \begin{equation}\label{convex00}
		C_0(tx_0+(1-t)y_{0})=q+\sum_{i=1}^{k} C_i(tx_i+(1-t)y_{i}),\\
	\end{equation}
	As, $t\x+(1-t)\y\geq {\bf 0}$, we need to show only the  complementarity conditions holds. From (\ref{u0v1}), $x_0*y_1=y_0*x_1=0$, we have \begin{equation}\label{convex11}
		(tx_0+(1-t)y_{0})*(tx_1+(1-t)y_{1})=t(1-t)x_0*y_1+t(1-t)y_0*x_1=0.\end{equation}
	For $i\in[k-1],$ from (\ref{di}), \begin{equation}\label{convex12}
		\begin{aligned}
			(d_i-(tx_i+(1-t)y_{i}))&*(tx_{i+1}+(1-t)y_{i+1})\\
			=&(t(d_i-x_i)+(1-t)(d_i-y_i))*(tx_{i+1}+(1-t)y_{i+1})\\
			=&t(1-t)((d_i-y_i)*x_{i+1}+(d_i-x_i)*y_{i+1})\\
			=&0.
		\end{aligned}
	\end{equation}
	Thus, from Eqs. (\ref{convex00}), (\ref{convex11}) and (\ref{convex12}), all the necessary complementarity conditions are satisfied, $t\x+(1-t)\y\in{\rm SOL}({\bf C},\d,q).$ Hence $\text{\rm SOL}({\bf C},\d,q)$ is convex for every $q\in\mathbb{R}^n$ and $\d \in  \Lambda^{(k-1)}_{n,++}$.\end{proof}

In LCP theory \cite{LCP}, if $C \in \mathbb{R}^{n\times n}$ is a column-sufficient matrix, then LCP($C, q$) has a unique solution for every $q > 0$. Next, we show a similar result for the cS-$W$ property. Although the proof is straightforward, we provide it here for sake of completeness.
\begin{theorem}
	Let ${\bf C}=(C_0,C_1,...,C_k)\in \Lambda^{(k+1)}_{n\times n}$ has the cS-$W$ property. If $C_0$ is a $M$-matrix, then for every $q \in \Rr_{++}$ and for every $\d \in  \Lambda^{(k-1)}_{n,++}$, $\text{\rm EHLCP}({\bf C},\d,q)$ has a unique solution.
\end{theorem}
\begin{proof}
	Let  $q \in \Rr_{++}$ and $\d=(d_{1}, d_{2},...,d_{k-1}) \in  \Lambda^{(k-1)}_{n,++}$. As $C_0$ is a $M$ matrix and $q \in \Rr_{++}$, we have $C_0^{-1}q> 0$. Let $\x=(C_0^{-1}q,0,...,0).$  By observation  $\x=(C_0^{-1}q,0,...,0)\in\text{SOL}({\bf C},\d,q).$ We clam that $\text{SOL}({\bf C},\d,q)=\{(C_0^{-1}q,0,...,0)\}.$
	
	Suppose ${\bf y}=(y_{0},y_{1},...,y_{k})\in \Lambda^{(k+1)}_n$ is an another solution to EHLCP(${\bf C},\d,q$). Then, 
	\begin{equation}\label{ssm}
		\begin{aligned}
			C_0 y_{0}=q+\sum_{i=1}^{k} C_iy_{i},~
			y_{0}\wedge y_{1}=0, ~ (d_{j}-y_{j})\wedge y_{j+1}=0~ \forall j\in[k-1].
		\end{aligned}
	\end{equation}
	From the Lemma \ref{l1}, we have
	\begin{equation}\label{unique}
		C_0y_{0}=q+\sum_{i=1}^{k} C_iy_{i}~\text{and}~y_{0}\wedge y_{j}=0~\forall~j\in [k].
	\end{equation}
	We let ${\bf{z}}:=\y-\x$,  then ${\bf{z}}=(y_{0}-C_0^{-1}q, y_{1},y_{2},...,y_{k})$. As $y_j\geq 0~\forall j\in[k]$, $z_i*z_j\geq 0~\forall1\leq i<j\leq k$. Also by a direct implication, from  (\ref{unique}), we get
	\begin{equation}\label{c0}
		C_0 (y_{0}-C_0^{-1}q)=\sum_{i=1}^k C_i y_{i}	\end{equation}
	and 
	$$ y_{j}\geq 0,~~(y_{0}-C_0^{-1}q)*y_{j}=y_{0}*y_{j}-C_0^{-1}q*y_{j}=-C_0^{-1}q*y_{j}\leq 0~\forall j\in [k].$$ As  $\bf C$ has the cS-$W$ property, we have $z_0*z_1 =0,$ this implies $$(y_{0}-C_0^{-1}q)*y_{1}=0\implies -C_0^{-1}q*y_{1}=0 $$ As $C_0^{-1}q>0, y_1=0.$ From complementarity property of solution $\y$, we have $(d_{j}-y_{j})\wedge y_{j+1}=0~ \forall j\in[k-1].$ As $y_1=0, (d_{1}-0)\wedge y_{2}=0\implies y_2=0.$ Similarly, $y_j=0~\forall j\in[k].$ 
	As $C_0$ is invertible, then from (\ref{c0}), $(y_{0}-C_0^{-1}q)=0.$  This gives ${\bf z}={\bf 0}.$ Hence ${\bf y}=\x.$ We have our conclusion.
\end{proof}
\section{Some relation between Column $W$ and cS-$W$ properties}
In LCP \cite{LCP}, we have a significant implication that follows:
$${\rm P~ matrix \iff Column\text{-}Sufficient~matrix~+~Nondegenerate ~matrix}.$$ It is interesting to observe a similar type of result for a set of matrices in the context of the EHLCP. To proceed further, we first prove a column representative result related to the  column ND-$W$ property. 

For a nondegenerate matrix, there is a nonzero principal minor result, which states that a matrix is nondegenerate if and only if all the principal minors of the given matrix are nonzero, see \cite{LCP}. Also, for a real matrix \( C \in \R \), all the principal minors are simply the determinants of the column representative matrices of \( (I, C) \). Motivated by this, we prove a similar result for the column ND-$W$ property.
\begin{theorem}\label{t1}Let ${\bf C}=(C_0,C_1,...,C_k)\in\Lambda^{(k+1)}_{n\times n}.$ Then, the following are equivalent.
	\begin{itemize}
		\item [\rm(i)] ${\bf C}$ has the column ND-$W$  property.
		\item[\rm(ii)] The determinant of every column representative matrix of ${\bf C}$ is nonzero.
	\end{itemize}
\end{theorem}
\begin{proof} (i)$\implies$(ii): 
	Let \({\bf C}\) has the column ND-$W$ property. Assume the contrary. Suppose there exists a column representative matrix \( M \) of \({\bf C}\) such that \(\text{det}(M) = 0\). According to Remark \ref{Represent}, there exists a set of  $n^{th}$ order diagonal matrices \( I_0, I_1, \ldots, I_k \) with entries of \(0\) or \(1\) only, satisfying \( I_i I_j = 0 \) for all \( 0 \leq i < j \leq k \) and \(\text{diag}(I_0 + I_1 + \cdots + I_k) \neq 0\) which satisfies
	$$ M = C_0 I_0 + C_1 I_1 + \cdots + C_k I_k.$$ Since \(\text{det}(M) = 0\), there exists a nonzero vector \( y \in \Rr \) such that 
	$$(C_0 I_0 + C_1 I_1 + \cdots + C_k I_k) y = 0.$$
	Define \( x_0 = -I_0 y \) and \( x_i = I_i y \) for all \( i \in [k] \). Then, we have
	$$\begin{aligned}  \begin{array}{r} C_0x_{0}=\sum_{i=1}^{k} C_ix_{i},~x_i*x_j=0~\forall0\leq i<j\leq k. \end{array} \end{aligned}$$
	By the column ND-$W$ property of ${\bf C},~\x=(-I_0y,I_1y,...,I_ky)={\bf 0}.$ As $\diag(I_0+I_1+\cdots+I_k)\neq 0$ and $I_iI_j=0~\forall 0\leq i<j\leq k$, we get that $y=0.$ Thus, we get a contradiction. \\ (ii)$\implies$(i): Suppose there exists a nonzero vector $\x=(x_0,x_1,...,x_k)\in\Lambda^{(k+1)}_{n}$ satisfies 
	\begin{equation}\label{1.0}\begin{aligned}  \begin{array}{r} C_0x_{0}=\sum_{i=1}^{k} C_ix_{i},~x_i*x_j=0~\forall0\leq i<j\leq k. \end{array} \end{aligned}\end{equation} 
	Let  $v\in\Rr$ be a vector whose $j^{th}$ component is given as $$v_j=\begin{cases}
		-(x_{0})_j&\text{ if }~ (x_{0})_j\neq 0\\  (x_{i})_j &~\text{ if } (x_{0})_j=0 ~\text{and} ~(x_{i})_j\neq 0~\text{for some}~i\in[k]\\
		0&\text{ if }~ (x_{0})_j=0 ~\text{and} ~(x_{i})_j= 0 ~\text{for all}~i\in[k]
	\end{cases},$$
	and  let $\{E_i\}^{k}_{i=1}$ be the  set of  $n^{th}$ order diagonal matrices,  defined by $$(E_0)_{jj}=\begin{cases}
		1&\text{ if }	(x_0)_j\neq0\\
		0& \text{ if }~ (x_{0})_j=0 ~\text{and} ~(x_{i})_j\neq 0 ~\text{for some}~i\in[k]\\
		1&\text{ if }~ (x_{0})_j=0 ~\text{and} ~(x_{i})_j= 0 ~\text{for all}~i\in[k]
	\end{cases},$$ and for $i\in[k],$ $$(E_{i})_{jj}=\begin{cases}
		0&\text{if } (x_{i})_j=0\\ 1&\text{ if } (x_{i})_j\neq 0
	\end{cases}.$$ By an easy verification, $v\neq 0$, $x_0=-E_0v$ and $x_i=E_iv~\forall i\in[k].$ Also, from Remark \ref{Represent},  $C_0E_0+C_1E_1+\cdots+C_kE_k$ is a column representative matrix of ${\bf C}$. Thus, from (\ref{1.0}), we have $$(C_0E_0+C_1E_1+\cdot\cdot\cdot+C_kE_k)v=0.$$ This implies that det$(C_0E_0+C_1E_1+\cdot\cdot\cdot+C_kE_k)=0.$ This gives us a contradiction.  Hence ${\bf C}$ has the column ND-$W$  property. \end{proof}
Now, we present our main result ot this section.
\begin{theorem}\label{4.20}
	Let ${\bf C}=(C_0,C_1,...,C_k)\in \Lambda^{(k+1)}_{n\times n}$. Then, the following are equivalent.
	\begin{itemize}
		\item [\rm(i)] ${\bf C}$ has the column $W$-property.
		\item [\rm(ii)] ${\bf C}$ have the cS-$W$ property and column ND-$W$ property.
		\item [\rm(iii)] ${\bf C}$ have the column $W_0$-property and column ND-$W$ property.
	\end{itemize}
\end{theorem}
\begin{proof}
	(i)$\implies$(ii): Let \({\bf C}\) has the column \(W\)-property. Then, from Definition \ref{cw0} and Theorem \ref{t1}, \({\bf C}\) also has the column ND-\(W\) property. To show ${\bf C}$  has the  cS-\(W\) property, we proceed with a contrary argument. Suppose there exists $\x=(x_0,x_1,...,x_k)\in\Lambda^{(k+1)}_{n}$ and $s\in \{0,1,2,...,k-1\}$ satisfies 
	\begin{equation}\label{con21} C_0x_{0}=\sum_{i=1}^{k} C_ix_{i},~x_i*x_j\geq 0~\forall1\leq i<j\leq k,~ x_0*x_i\leq 0~\forall i\in[k] \text{ and }x_s*x_{s+1}\neq 0.
	\end{equation}
	We now construct a vector $y\in\Rr$ whose $j^{\rm{th}}$ component is given by $$y_j=\begin{cases}
		-1 & \text{if} ~(x_{0})_j>0\\ 1 & \text{if}  ~(x_{0})_j<0\\1 & \text{if} ~(x_{0})_j=0 ~\text{and}~(x_{i})_j> 0~\text{for some}~ i\in[k]\\-1 & \text{if} ~(x_{0})_j=0 ~\text{and}~(x_{i})_j< 0~\text{for some}~ i\in[k]\\0 & \text{if}~ (x_{0})_j=0 ~\text{and} ~(x_{i})_j= 0 ~\text{for all}~i\in[k] 
	\end{cases}.$$
	Since $x_s*x_{s+1}\neq 0$, ${{y}}$  is a nonzero vector. Next, we construct a set of nonnegative  diagonal matrices $D_{0}, D_{1},...,D_{k}$ which are defined by $$(D_{0})_{jj}=\begin{cases}|(x_{0})_j| & \text{if} ~(x_{0})_j\neq0\\ 0 & \text{if} (x_{0})_j=0~\text{and}~(x_{i})_j\neq 0~\text{for some}~ i\in[k] \\1 & \text{if}~ (x_{0})_j=0 ~\text{and} ~(x_{i})_j= 0 ~\text{for all}~i\in[k]
		,\end{cases}$$ and  for $i\in [k]$,
	$$(D_{i})_{jj}= |(x_{i})_j|.$$
	By an easy verification we can see that  $\text{diag}(D_{0}+ D_{1}+\cdots+D_{k})>0$ with  
	\begin{equation}\label{22}
		x_{0}=-D_{0}y~\text{and}~x_{i}=D_{i}y~\forall i\in [k].
	\end{equation}
	Substituting (\ref{22}) in $C_0x_{0}=\sum_{i=1}^{k} C_ix_{i}$, we get \begin{equation*}
		C_0(-D_{0}y)=\sum_{i=1}^{k} C_iD_{i}(y) \Rightarrow 
		\big(C_0D_{0}+C_1D_{1}+\cdots+C_kD_{k}\big)y=0.
	\end{equation*} 
	This implies that det$(C_0D_{0}+C_1D_{1}+\cdots+C_kD_{k}\big)=0$. Since $x\neq 0$, ${\bf C}$ does not have the column $W$-property from Theorem \ref{cw0}. Thus, we get a contradiction. Therefore, ${\bf C}$ has the cS-$W$ property. \\
	(ii)$\implies$(iii): 	Let  ${\bf C}=(C_0,C_1,...,C_k)\in \Lambda^{(k+1)}_{n\times n}$. It is enough to prove ${\bf C}$ has the column $W_0$-property.  Let us assume a set of positive diagonal matrices $\{D_i\}_{i=0}^{k}$ of order $n$ and a vector $x\in\Rr$ such that   
	\begin{equation}\label{w0}
		({C_0D_0+C_1D_1+\cdots+C_kD_k})x=0.
	\end{equation}
	Let $
	x_0=-D_0x\text{ and }x_i=D_ix~\forall 1\leq i\leq k,$ Then, substituting these into (\ref{w0}), we get
	$$C_0x_{0}=\sum_{i=1}^{k} C_ix_{i}.$$
	Also, 
	$$x_i*x_j=D_ix*(D_jx)=D_iD_jx^2\geq 0~\forall1\leq i<j\leq k,$$
	and 
	$$x_0*x_i=-D_0x*D_ix=-D_0D_ix^2\leq 0~\forall i\in[ k].$$	
	As ${\bf C}$ has the  cS-$W$ property, from the above, $x_i*x_{i+1}=0~\forall 0\leq i\leq k-1$. This implies that $D_ix*D_{i+1}x=0.$ Since $D_i$ is a positive diagonal matrix for $0\leq i\leq k$, $x=0.$ Therefore,  ${\bf C}$ has the column $W_0$-property. \\
	(iii)$\implies$(i): Suppose \({\bf C}\) have the column \(W_0\) and column ND-\(W\) properties. From Definition \ref{cw0}, the column \(W_0\)-property ensures that the determinants of all the column representative matrices are nonnegative (nonpositive). Additionally, from Theorem \ref{t1}, the column ND-\(W\) property implies that the determinants of all the column representative matrices are nonzero. Therefore, the determinants of all the column representative matrices are positive (negative). Hence, \({\bf C}\) has the column \(W\)-property.
\end{proof}
In LCP \cite{LCP}, for a real matrix, we have the  following implication:\begin{equation}\label{19.00}
	P\subseteq \text{Column~Sufficient}\subseteq P_0.\end{equation}
From Theorem \ref{4.20}, we have a direct result as an extension of  implication (\ref{19.00}).
\begin{theorem}\label{4.2}
	Let  ${\bf C}=(C_0,C_1,...,C_k)\in \Lambda^{(k+1)}_{n\times n}$. Consider the following statements:
	\begin{itemize}
		\item [\rm(i)] ${\bf C}$ has the column $W$-property
		\item [\rm(ii)] ${\bf C}$ has the cS-$W$ property.
		\item [\rm(iii)] ${\bf C}$ has the column $W_0$-property
	\end{itemize}
	Then, $\rm(i)\implies\rm(ii)\implies\rm(iii).$
\end{theorem}
In above theorem, implication (ii)$\implies$(i) will not hold.  To show this we present the following example.
\begin{example}\rm
	Let ${\bf C}=(C_0,C_1,C_2)\in\Lambda^{(3)}_{2\times 2}$ where $$C_0=C_2=\begin{bmatrix}
		1&0\\0&1\\
	\end{bmatrix},~C_2=\begin{bmatrix}
		0&1\\-1&0
	\end{bmatrix}.$$ First we show that ${\bf C}$ has the cS-$W$ property. Let $\x=(x,y,z)\in\Lambda^{(3)}_{2}$ such that 
	$$C_0x=C_1y+C_2z,~x*y\leq 0,x*z\leq 0\text{ and } y*z\geq 0,$$ where $x=(x_1,x_2),y=(y_1,y_2)$ and $z=(z_1,z_2)$ are vectors in $\Rr.$
	From this, we have $x_i y_i\leq0,x_i z_i\leq0~\forall i\in\{1,2\}$ and 
	\begin{equation}
		\begin{aligned}
			x_1=y_2+z_1,~~x_2=-y_1+z_2,\\
		\end{aligned}
	\end{equation}
	Now multiplying both $x_1$ and $x_2$ with $y_1$ and $y_2$ respectively, we get
	\begin{equation}
		\begin{aligned}\label{3.2.6}
			x_1y_1=y_1y_2+z_1y_1\leq 0,~~
			x_2y_2=-y_1y_2+z_2y_2\leq 0\\
		\end{aligned}
	\end{equation}
	
	As $x_i y_i\leq0~\forall i\in\{1,2\}$, we have $x_1y_1+x_2y_2\leq 0$. This implies
	
	$$y_1y_2+z_1y_1-y_1y_2+z_2y_2\leq 0\implies y_1z_1+y_2z_2\leq 0.$$
	Since $y*z\geq 0$, $y_1z_1$ and $y_2z_2$ will be non negative. Therefore, $$y_1z_1+y_2z_2=0\implies y_iz_i=0~\forall i\in\{1,2\}.$$ This gives   $x_1y_1+x_2y_2=0\implies x_1y_1=x_2y_2=0$. Thus $x_iy_i=y_iz_i=0$ for $i\in\{1,2\}$.
	This ensures that ${\bf C}$ has the cS-$W$ property. 
	
	Now we take  a column representative matrix of ${\bf C}$ which is 
	$$R=[{(C_0)}_{.1}, {(C_2)}_{.2}]=\begin{bmatrix}
		1&1\\0&0\\
	\end{bmatrix}.$$ By observation, det$(R)$=0. Hence ${\bf C}$ does not have the column $W$-property. Hence the cS-$W$ property does not imply column $W$-property.
\end{example}
\begin{rem}\rm
	Similar to previous example, the column $W_0$-property does not imply the cS-$W$ property. By taking $k=3$ and $C_0=I,C_1=C_2=0$, we can see that converse of implication (ii)$\implies$(iii) in Theorem 4.2 is not true.  
\end{rem}
For a $Z$-matrix [A matrix with off-diagonal entries are nonpositive], the column-sufficient matrix property is equivalent to the column-sufficient matrix property on \( \Rr_+ \), see \cite{LCP}. For a set of $Z$-matrices, we prove a similar result for the cS-$W$ property. First, we define the cone cS-$W$ property. Then, we establish an equivalence between the cS-$W$ property and the cone cS-$W$ property for $Z$-matrices.
\begin{definition}\label{ccSW}
	We say ${\bf C}=(C_0,C_1,...,C_k)\in\Lambda^{(k+1)}_{n\times n}$ has cone column sufficient-W (cone cS-$W$) property if 
	\begin{equation*}
		\big[C_0x_{0}=\sum_{i=1}^{k} C_ix_{i},~x_{i}\geq 0~\text{and}~ x_{0}* x_{i}\leq 0~\forall i\in[k]~\big]\Rightarrow  x_i*x_{i+1}=0~\forall 0\leq i\leq k-1.
	\end{equation*}
\end{definition}
\begin{theorem}\label{cssm}
	Let ${\bf C}=(C_0,C_1,...,C_k)\in \Lambda_{n\times n}^{(k+1)}$ such that $C_0^{-1}C_i$ be a $Z$-matrix for all $i\in[k]$. Then the following statements are equivalent.\begin{itemize}
		\item [\rm (i)] ${\bf C}$ has the \cs property.
		\item[\rm(ii)]  ${\bf C}$ has the cone \cs property.
	\end{itemize}
\end{theorem}
\begin{proof}
	(i)$\implies$(ii): It follows from Definitions \ref{csww} and  \ref{ccSW}.\\
	(ii)$\implies$(i): Assume ${\bf C}$ has the cone \cs property. Let $\x=(x_0,x_1,...,x_k)\in\Lambda^{(k+1)}_{n}$ such that \begin{equation}\label{con} x_i*x_j\geq 0~\forall1\leq i<j\leq k~\text{ and } C_0x_{0}=\sum_{i=1}^{k} C_ix_{i},~ x_0*x_i\leq 0~\forall i\in[k].	
	\end{equation}
	As $C_0$  is invertible, from \ref{con}, we have \begin{equation}\label{con12}
		x_{0}=\sum_{i=1}^{k} C_0^{-1}C_ix_{i}.\end{equation}
	Consider $y_0\in\Rr$ such that
	\begin{equation}\label{zp}
		y_{0}=\sum_{i=1}^{k} C_0^{-1}C_i|x_{i}|.\end{equation} For some coordinate $r\in[n]$\begin{equation*}\begin{aligned}
			(y_{0}*{(|x_j|)})_r&= (\sum_{i=1}^{k} C_0^{-1}C_i|x_{i}|)_r{(|x_j|)}_r\\
			&=\sum_{i=1}^{k} (C_0^{-1}C_i)_{rr}(|x_{i}|)_r{(|x_j|)}_r+\sum_{i=1}^{k}\sum_{r\neq s,s=1}^{n} (C_0^{-1}C_i)_{rs}(|x_{i}|)_s{(|x_j|)}_r.\\
		\end{aligned}
	\end{equation*}
	
	As $C_0^{-1}C_i$ is a $Z$ matrix, then the off diagonal entries of $C_0^{-1}C_i$ will be nonpositive. Since $x_{i}*x_j\geq 0 ~\forall1\leq i,j\leq k$, we have 
	\begin{equation*}\begin{aligned}
			(y_{0}*{(|x_j|)})_r
			\leq&\sum_{i=1}^{k} (C_0^{-1}C_i)_{rr}{(x_{i})}_r{(x_j)}_r+\sum_{i=1}^{k}\sum_{r\neq s,s=1}^{n} (C_0^{-1}C_i)_{rs}{(x_{i})}_s{(x_j)}_r\\
			=& (\sum_{i=1}^{k} C_0^{-1}C_ix_{i})_r{(x_j)}_r\\
		\end{aligned}.
	\end{equation*}
	From (\ref{con12}), we get that $$	y_{0}*{(|x_j|)}\leq x_0*x_j\leq 0. $$
	As ${\bf C}$ has the cone \cs property, then for $\y=(y_0,y_1,...,y_k)=(y_0,|x_1|,|x_2|,,,|x_k|),$ we have $$y_j*y_{j+1}=0,0\leq j\leq k-1.$$ For $j=0,$ $y_0*|x_1|=0$. Then from (\ref{zp}),$$0=y_0*|x_1|\leq x_0*x_1\leq 0\implies  x_0*x_1=0.$$ For $j\in[k-1],$ $$|x_i|*|x_{i+1}|=0\implies  x_i*x_{i+1}=0.$$ 
	Therefore $x_i*x_{i+1}=0,0\leq i\leq k-1.$ Hence {\bf C} has the \cs-property.
\end{proof}
From Theorem \ref{convex} and Theorem \ref{cssm}, we have the following result.
\begin{corollary}
	Let ${\bf C}=(C_0,C_1,...,C_k)\in \Lambda_{n\times n}^{(k+1)}$ such that $C_0^{-1}C_i$ be a $Z$ matrix for all $i\in[k]$. If  ${\bf C}$ has the cone \cs-property, then   $\text{\rm SOL}({\bf C},\d,q)$ is convex for every $q\in\mathbb{R}^n$ and $\d \in  \Lambda^{(k-1)}_{n,++}$.
\end{corollary}

\section{Conclusion}
In this article, we introduce the column sufficient-$W$ property and prove a convexity result for the solution set of the EHLCP. We also establish a uniqueness result for \( q > 0 \) in the presence of the cS-$W$ property. Additionally, we derive some relations between the column-$W$ ($W_0$) property and the cS-$W$ property with the help of the column ND-$W$ property. Lastly, we present a result for cS-$W$ property on $\Rr_{+}$. 

{\small \bf Conflict of Interest}: {\small The authors have no conflict of interests.}


\begin{thebibliography}{99}
		\bibitem{wcp}{Chi, X., Gowda, M.S., Tao, J.: The weighted horizontal linear complementarity problem on a Euclidean Jordan algebra. J. Global Optim. {\bf73}, 153-169 (2019)}
	\bibitem{sufficient}Cottle, R. W., Pang, J. S., $\&$ Venkateswaran, V. Sufficient matrices and the linear complementarity problem. Linear Algebra and its applications, 114, 231-249, (1989).
	\bibitem{glcp}{Cottle, R.W., Dantzig, G.B.: A generalization of the linear complementarity problem. J. Comb. Theory {\bf8}(1), 79-90 (1970).}
	\bibitem{LCP} {Cottle, R.W., Pang, J.S., Stone, R.E.: The Linear Complementarity Problem. Academic Press, New York (1992).}
	\bibitem{cnd}{Gowda, M. Seetharama. "On the extended linear complementarity problem." Mathematical Programming 72, 33-50 (1996).}
	\bibitem{telcp}{De Schutter B., De Moor B.: The extended linear complementarity problem. Math. Program. {\bf 71}(3), 289-325 (1995).}
	\bibitem{deg}{Gowda, M.S.: Applications of degree theory to linear complementarity problems. Math. Oper. Res. {\bf18}, 868-879 (1993).}
	\bibitem{pk}{P.K. Yadav,  K. Palpandi, Generalizations of $R_0$	and {\bf SSM}	Properties for Extended Horizontal Linear Complementarity Problem. J Optim Theory Appl.  199, 392–414 (2023). \url{https://doi.org/10.1007/s10957-023-02262-9}}
	\bibitem{cnd}{Yadav, Punit Kumar, and Palpandi Karuppaiah. \textquotedblleft On finiteness of the solution set of extended horizontal linear complementarity problem. Operations Research Letters 54 (2024).
}
	\bibitem{elcp}{Mangasarian, O.L., Pang, J.-S.: The extended linear complementarity problem. SIAM J. Matrix Anal. Appl.    {\bf16}(2), 359-368 (1995).}
	\bibitem{homo}{Ralph, D.: A stable homotopy approach to horizontal linear complementarity problems. Control Cybern.  {\bf31}(3), 575-600 (2002).}
	\bibitem{szn}{Sznajder, R.: Degree-theoretic analysis of the vertical and horizontal linear complementarity problems, Ph.D. Thesis, University of Maryland Baltimore County (1994).}
	\bibitem{PP0}{Sznajder, R., Gowda, M.S.: Generalizations of $P_0$- and $P$-properties; extended vertical and horizontal linear complementarity problems. Linear Algebra Appl. {\bf223-224}, 695-715 (1995).}
\end{thebibliography}
\end{document}